\documentclass[a4paper]{article}
\usepackage{amsmath, amssymb, amsthm} 
\usepackage{type1cm}
\usepackage{graphicx}
\usepackage[all]{xy}
\theoremstyle{plain}
\newtheorem{theorem}{Theorem}
\newtheorem{proposition}{Proposition}
\newtheorem{lemma}{Lemma}
\newtheorem{corollary}{Corollary}
\theoremstyle{definition}

\newtheorem{definition}{Definition}
\newtheorem{example}{Example}
\newtheorem{lexample}{Logical Example}
\theoremstyle{remark}

\newcommand{\bbbn}{\mathbb{N}}

\newcommand{\bbbr}{\mathbb{R}}

\title{A Model Theory Approach to Structural Limits\thanks{This work, which appeared in Comment. Math. Univ. Carolin. 53,4 (2012) 581--603, is supported by grant ERCCZ LL-1201 of the Czech Ministry of Education and CE-ITI of GA{\v C}R}}
\author{J. Ne\v set\v ril \\
\small Computer Science Institute of Charles University (IUUK and ITI)\\
 \small    Malostransk\' e n\' am.25, 11800 Praha 1, Czech Republic\\
 \small    {\tt nesetril@kam.ms.mff.cuni.cz}
 \\
 \\
\and P. Ossona de Mendez\\
\small  Centre d'Analyse et de Math\'ematiques Sociales (CNRS, UMR 8557)\\
 \small   190-198 avenue de France, 75013 Paris, France\\
\small {\tt pom@ehess.fr}
}
\date{\small (appeared in Comment. Math. Univ. Carolin. 53,4 (2012) 581--603)}
 \begin{document}
 \maketitle
 \begin{abstract}
The goal of this paper is to unify two lines in a particular 
area of graph limits.
First, we generalize and provide unified treatment of various
graph limit concepts by means of a combination of  model theory and analysis. 
Then, as an example, we generalize limits of bounded degree graphs from subgraph testing to 
finite model testing. 
 \end{abstract}
\section{Introduction}
\label{se:intro}
Recently,  graph sequences and
graph limits are intensively studied, from diverse point of views: 
probability theory and statistics, property testing in computer science, flag algebras,
logic, graphs homomorphisms, etc.
Four standard notions of graph limits have inspired this work:
\begin{itemize}
  \item the notion of {\em dense graph limit} \cite{Borgs2005,Lov'asz2006};
  \item the notion of {\em bounded degree graph limit} \cite{Benjamini2001,Aldous2006};
  \item the notion of {\em elementary limit} e.g. \cite{Hodges1997,Lascar2009}
  \item the notion of {\em left limit} developed by the authors
  \cite{ECM2009,Sparsity}.
\end{itemize}

Let us briefly introduce these notions. Our combinatorial terminology is standard and we refer to the standard books 
(such as \cite{Hodges1997,Invitation1998,Sparsity,Rudin1973}) 
or original papers for more information.

The first approach consists in randomly picking a mapping from a test graph and to 
 check whether this is a homomorphism. A sequence $(G_n)$ of graphs
 will be said to be {\em L-convergent} if 
 $$t(F,G_n)=\frac{{\rm hom}(F,G_n)}{|G_n|^{|F|}}$$
 converges for every fixed (connected) graph $F$.
 
 The second one is used to define the convergence of a sequence of graphs with
 bounded degrees. A sequence $(G_n)$ of graphs with bounded maximum degrees
 will be said to be {\em BS-convergent} if, for every integer $r$, the 
 probability that the ball of radius $r$ centered at a random vertex of $G_n$ is
 isomorphic to a fixed rooted graph $F$ converges for every $F$.

The third one is a general notion of convergence based on the first-order
properties satisfied by the elements of the sequence. A sequence
$(G_i)_{i\in\bbbn}$ is {\em elementarily convergent} 
 if,
for every sentence $\phi$ there exists an integer $n_\phi$ such that either all the $G_i$ with
$i>n_\phi$ satisfy $\phi$ or none of them do.

The fourth notion of convergence is based in testing existence of homomorphisms
from fixed graphs: a sequence $(G_n)$ is said to be {\em left-convergent} if,
for every graph $F$, either all but a finite number of the graphs $G_n$  contain
a homomorphic image of $F$ or only a finite number of $G_n$ does. In other
words, left-convergence is a weak notion of elementary convergence where
we consider primitive positive sentences only. 
 
 These four notions proceed in different directions and, particularly, relate
 to either dense or sparse graphs. The sparse--dense dichotomy seems to be a key
 question in the area.

In this paper we provide a unifying approach to these limits. Our approach is a combination of
a functional analytic and model theoretic approach and thus applies to applies to more general structures (rather than graphs). Thus we use term {\em structural limits}.

The paper is organized as follows: In Section~\ref{sec:bool} we briefly introduce a general machinery based on the 
Boolean algebras and dualities,
see \cite{Halmos1998} for standard background material.
In section~\ref{sec:logic} we apply this to Lindenbaum-Tarski algebras to get a representation of limits as measures (Theorem~\ref{thm:bamu}).
In section~\ref{sec:nonstd} we mention an alternative approach by means of ultraproducts (i.e. a non-standard approach) which yields 
another representation (of course ineffective) of limits (Proposition~\ref{prop:hlim}).
In section~\ref{sec:case} we relate this to examples given in this section and 
particularly state results for bounded degree graphs, thus extending Benjamini-Schramm convergence \cite{Benjamini2001}
to the general setting of ${\rm FO}$-convergence (Theorem~\ref{thm:Bounded}).
In the last section, we discuss the type of limit objects we would like to construct, and introduce
some applications to the study of particular cases of first-order convergence which are
going to appear elsewhere.
   
\section{Boolean Algebras, Stone Representation, and Measures}
\label{sec:bool}
Recall that a {\em Boolean algebra} $B$ is an algebra  
with two binary operations $\vee$ and $\wedge$, a unary operation $\neg$
 and two elements $0$ and $1$, such that
$(B,\vee,\wedge)$ is a distributive lattice with minimum $0$ and maximum $1$ 
which is complemented (in the sense that the complementation $\neg$ satisfies $a\vee \neg a=1$ and $a\wedge\neg a=0$). 

The smallest Boolean algebra, denoted $\mathbf 2$, has elements $0$ and $1$. In this Boolean algebra it holds
$0\wedge a=0, 1\wedge a=a, 0\vee a=a, 1\vee a=1, \neg 0=1$, and $\neg 1=0$.
Another example is
the powerset $2^X$ of a set $X$ has a natural structure of Boolean algebra, with
$0=\emptyset, 1=X$, $A\vee B=A\cup B$, $A\wedge B=A\cap B$ and $\neg A=X\setminus A$.

Key examples for us are the following:

\begin{lexample}
The class of all first-order formulas on a language $\mathcal L$, considered up to logical equivalence,
form a Boolean algebra  with conjunction $\vee$, disjunction $\wedge$ and negation $\neg$ and constants
``false'' ($0$) and ``true" ($1$). This Boolean algebra will be denoted ${\rm FO}(\mathcal L)$.

Also, we denote by ${\rm FO}_0(\mathcal L)$ the Boolean algebra
of all first-order sentences (i.e. formulas without free variables) on a language $\mathcal L$,
 considered up to logical equivalence. Note ${\rm FO}_0(\mathcal L)$ is a Boolean sub-algebra of
 ${\rm FO}(\mathcal L)$. 
\end{lexample}

\begin{lexample}
\label{lex:1}
Consider a logical theory $T$ (with negation).
The {\em Lindenbaum-Tarski} algebra $L_T$ of $T$ consists of the equivalence classes of sentences 
of $T$ (here two sentences $\phi$ and $\psi$ are equivalent
 if they are provably equivalent in $T$).
The set of all the first-order formulas that are provably false
from $T$ forms an ideal $\mathcal I_T$ of the Boolean algebra ${\rm FO}_0(\mathcal L)$ and
$L_T$ is nothing but the quotient algebra ${\rm FO}_0(\mathcal L)/\mathcal I_T$.
\end{lexample}

With respect to a fixed Boolean algebra $B$, a {\em Boolean function}
is a function obtained by a finite combination of the operations
$\vee$, $\wedge$, and $\neg$.

Recall that a function $f:B\rightarrow B'$
 is a {\em homomorphism} of Boolean algebras if
 $f(a\vee b)=f(a)\vee f(b)$, $f(a\wedge b)=f(a)\wedge f(b)$, $f(0)=0$ and $f(1)=1$. 
A filter of a Boolean algebra $B$ is an upper set $X$ (meaning that $x\in X$ and $y\geq x$ imply $y\in X$)
that is a proper subset of $B$ and that is closed under $\wedge$ operation ($\forall x,y\in X$ it holds $x\wedge y\in X$).
It is characteristic for Boolean algebras that the maximal filters
coincide with the {\em prime filters}, that is, the (proper) filters $X$ such that
$a\vee b\in X$ implies that either $a\in X$ or $b\in X$. One speaks of the
maximal (i.e. prime filters) as of {\em ultrafilters} (they are also characterized
by the fact that for each a either $a\in X$ or $\neg a \in X$). 
It is easily checked that the mapping
$f\mapsto f^{-1}(1)$ is a bijection between the homomorphisms $B\rightarrow \mathbf 2$ and the ultrafilters on $B$.

A {\em Stone space}  is a compact Hausdorff with a basis of clopen
subsets. With a Boolean algebra $B$ associate a topological space
$$S(B)=(\{x,\ x\text{ is a ultrafilter in }B\},\tau),$$
where $\tau$ is the topology generated by all the $K_B(b)=\{x,\ b\in x\}$ (the
subscript $B$ will be omitted if obvious). Then $S(B)$ is a Stone space.
By the well-known Stone Duality Theorem \cite{Stone1936},
the mappings $B \mapsto S(B)$ and $X\mapsto\Omega(X)$, where 
$\Omega(X)$ is the Boolean
algebra of all clopen subsets of a Stone space $X$, constitute a one-one 
correspondence between the classes of all Boolean algebras and all
Stone spaces.

In the language of category theory, Stone's representation theorem
means that there is a duality between the category of Boolean algebras
(with homomorphisms) and the category of Stone spaces (with continuous functions).
The two contravariant functors defining this duality are denoted by $S$ and $\Omega$
and defined as follows:

For every homomorphism $h:A\rightarrow B$ between two Boolean algebra,
we define the map $S(h):S(B)\rightarrow S(A)$ by $S(h)(g)=g\circ h$
(where points of $S(B)$ are identified with homomorphisms $g:B\rightarrow\mathbf 2$).
Then for every homomorphism $h:A\rightarrow B$, the map
$S(h):S(B)\rightarrow S(A)$ is a continuous function.
Conversely, for every continuous function $f:X\rightarrow Y$
between two Stone spaces, define the map $\Omega(f):\Omega(Y)\rightarrow\Omega(X)$
by $\Omega(f)(U)=f^{-1}(U)$ (where elements of $\Omega(X)$ are identified with
 clopen sets of $X$). Then for every continuous function 
 $f:X\rightarrow Y$, the map $\Omega(f):\Omega(Y)\rightarrow\Omega(X)$ is a 
 homomorphism of Boolean algebras.

We denote by $K=\Omega\circ S$ one of the two natural isomorphisms
defined by the duality. Hence, for a Boolean algebra $B$, 
$K(B)$ is the set algebra $\{K_B(b): b\in B\}$, and this 
algebra is isomorphic to $B$. 

Thus we have a natural notion for convergent sequence of elements of $S(B)$ 
(from Stone representation follows that this may be seen as the pointwise converegence).

\begin{lexample}
\label{lex:2}
Let $B={\rm FO}_0(\mathcal L)$ denote the Boolean Lindenbaum-Tarski algebra of all first-order sentences  
on a language $\mathcal L$ 
up to logical equivalence. Then the filters of $B$ are the consistent theories of
${\rm FO}_0(\mathcal L)$ and the ultrafilters of $B$ are the {\em complete theories} of ${\rm FO}_0(\mathcal L)$ (that
is maximal consistent sets of sentences). It follows that the closed sets of $S(B)$
correspond to finite sets of consistent theories.
According to G\"odel's completeness theorem, every consistent theory
has a model. It follows that the completeness theorem for first-order logic --- which states that 
that a set of first-order sentences has a model if and only if every finite subset of it has a model
--- amounts
to say that $S(B)$ is compact. The points of $S(B)$ can also be identified with {\em elementary equivalence} classes
of models. The notion of convergence of models induced by the topology of $S(B)$, called
{\em elementary convergence}, has been extensively studied.
\end{lexample}

An ultrafilter on a Boolean algebra $B$ can be considered as a finitely additive measure, 
for which every subset has either measure $0$ or $1$.
Because of the equivalence of the notions of Boolean algebra and of set algebra, we define
the {\em ba space} ${\rm ba}(B)$ of $B$ has the space of all bounded 
additive functions $f:B\rightarrow\bbbr$.
Recall that a function $f:B\rightarrow\bbbr$ is {\em additive} 
if for all $x,y\in B$ it holds
$$
x\wedge y=0\quad\Longrightarrow\quad f(x\vee y)=f(x)+f(y).
$$
The space ${\rm ba}(B)$ is a Banach space for the norm
$$
\|f\|=\sup_{x\in B} f(x) - \inf_{x\in B} f(x).
$$
(Recall that the ba space of an algebra of sets $\Sigma$ is the Banach space consisting of all bounded
 and finitely additive measures on $\Sigma$ with the total variation norm.) 

Let $h$ be a homomorphism $B\rightarrow\mathbf 2$ and let $\iota:\mathbf 2\rightarrow\bbbr$ be defined by
$\iota(0)=0$ and $\iota(1)=1$. Then $\iota\circ h\in{\rm ba}(B)$. Conversely, 
if $f\in{\rm ba}(B)$ is such that $f(B)=\{0,1\}$ then
$\iota^{-1}\circ f$ is a homomorphism $B\rightarrow\mathbf 2$.
 This shows that $S(B)$ can be identified 
with a subset of ${\rm ba}(B)$.  

One can also identify ${\rm ba}(B)$ with the space 
${\rm ba}(K(B))$ of finitely additive measure defined on the set algebra
$K(B)$. As  vector spaces ${\rm ba}(B)$ is isomorphic to ${\rm ba}(K(B))$ and thus  ${\rm ba}(B)$ 
is then 
clearly the (algebraic) dual
of the normed vector space $V(B)$ (of so-called {\em simple functions}) 
generated by the indicator functions 
of the clopen sets (equipped with supremum norm).
Indicator functions of clopen sets are denoted by
 $\mathbf 1_{K(b)}$ (for some $b\in B$) and defined by
$$\mathbf 1_{K(b)}(x)=\begin{cases}
1&\text{if }x\in K(b)\\
0&\text{otherwise.}
\end{cases}
$$ 
  
The pairing of a function $f\in {\rm ba}(B)$ and
a vector $X=\sum_{i=1}^n a_i \mathbf 1_{K(b_i)}$ is defined by
$$
[f,X]=\sum_{i=1}^n a_i f(b_i).
$$
That $[f,X]$ does not depend on a particular choice of a decomposition of $X$
follows from the additivity of $f$. We include a short proof for completeness:
Assume $\sum_i \alpha_i\mathbf 1_{K(b_i)}=\sum_i \beta_i\mathbf 1_{K(b_i)}$. As
for every $b,b'\in B$ it holds $f(b)=f(b\wedge b')+f(b\wedge\neg b')$
and $\mathbf 1_{K(b)}=\mathbf 1_{K(b\wedge b')}+\mathbf 1_{K(b\wedge \neg b')}$
we can express the two sums as
$\sum_j \alpha_j'\mathbf 1_{K(b_j')}=\sum_j \beta_j'\mathbf 1_{K(b_j')}$ (where
$b_i'\wedge b_j'=0$ for every $i\neq j$), with
$\sum_i \alpha_i f(b_i)=\sum_j \alpha_j' f(b_j')$ and
$\sum_i \beta_i f(b_i)=\sum_j \beta_j' f(b_j')$.
As $b_i'\wedge b_j'=0$ for every $i\neq j$, for $x\in K(b_j')$ 
it holds $\alpha_j'=X(x)=\beta_j'$. Hence $\alpha_j'=\beta_j'$ for every $j$.
Thus  $\sum_i \alpha_i f(b_i)=\sum_i \beta_i f(b_i)$.

Note that $X\mapsto [f,X]$ is indeed continuous. Thus ${\rm ba}(B)$ can
also be identified with the continuous dual of $V(B)$. We now show 
that the vector space $V(B)$ is dense in the space $C(S(B))$ of continuous
functions from $S(B)$ to $\bbbr$, hence that ${\rm ba}(B)$ can also
be identified with the continuous dual of $C(S(B))$:

\begin{lemma}
\label{lem:densv}
The vector space $V(B)$ is dense in $C(S(B))$ (with the uniform norm).
\end{lemma}
\begin{proof}
Let $f\in C(S(B))$ and let $\epsilon>0$. 
For $z\in f(S(B))$ let $U_z$ be the preimage by $f$ of the open ball $B_{\epsilon/2}(z)$ of $\bbbr$ centered in $z$.
As $f$ is continuous, $U_z$ is a open set of $S(B)$. As $\{K(b): b\in B\}$ is a basis of the topology of $S(B)$, 
$U_z$ can be expressed as a union $\bigcup_{b\in\mathcal F(U_z)}K(b)$. It follows that
$\bigcup_{z\in f(S(B))}\bigcup_{b\in\mathcal F(U_z)}K(b)$ is a covering of $S(B)$ by open sets. As $S(B)$ is compact,
there exists a finite subset $\mathcal F$ of $\bigcup_{z\in f(S(B))}\mathcal F(U_z)$ that covers $S(B)$.
Moreover, as for every $b,b'\in B$ it holds $K(b)\cap K(b')=K(b\wedge b')$ and $K(b)\setminus K(b')=K(b\wedge\neg b')$
it follows that we can assume that there exists a finite family $\mathcal F'$ such that
$S(B)$ is covered by open sets $K(b)$ (for $b\in \mathcal F'$) and such that for every $b\in \mathcal F'$ there exists
$b'\in\mathcal F$ such that $K(b)\subseteq K(b')$. In particular, it follows that for every $b\in\mathcal F'$, $f(K(b))$
is included in an open ball of radius $\epsilon/2$ of $\bbbr$. For each $b\in\mathcal F'$ choose a point $x_b\in S(B)$ such 
that $b\in x_b$. Now define
$$
\hat f=\sum_{b\in\mathcal F'}f(x_b)\mathbf 1_{K(b)}
$$   
Let $x\in S(B)$. Then there exists $b\in \mathcal F'$ such that $x\in K(b)$. Thus
$$|f(x)-\hat f(x)|=|f(x)-f(x_b)|<\epsilon.$$
Hence $\|f-\hat{f}\|_\infty<\epsilon$. 
\qed\end{proof}

It is difficult to exhibit a basis of $C(S(B))$ or
$V(B)$. However, every meet sub-semilattice
of a Boolean algebra $B$ generating $B$ contains 
(via indicator functions) a basis of $V(B)$:

\begin{lemma}
\label{lem:meetbasis}
Let $X\subseteq B$ be closed by $\wedge$ and such
that $X$ {\em generates} $B$ (meaning that every 
element of $B$ can be obtained 
as a Boolean function of finitely many elements
from $X$).

Then $\{\mathbf 1_b:\ b\in X\}\cup\{\mathbf 1\}$ 
(where $\mathbf 1$ is the constant function with value $1$) 
includes a basis of 
the vector space $V(B)$.
\end{lemma}
\begin{proof}
Let $b\in B$. As $X$ generates $B$ there exist
$b_1,\dots,b_k\in X$ and a Boolean function $F$
such that $b=F(b_1,\dots,b_k)$.
As $\mathbf 1_{x\wedge y}=\mathbf 1_x\,\mathbf 1_y$ and
$\mathbf 1_{\neg x}=\mathbf 1-\mathbf 1_x$ there exists a polynomial
$P_F$ such that $\mathbf 1_b=P_F(\mathbf 1_{b_1},\dots,\mathbf 1_{b_k})$.
For $I\subseteq [k]$, the monomial $\prod_{i\in I}\mathbf 1_{b_i}$
rewrites as $\mathbf 1_{b_I}$ where $b_I=\bigwedge_{i\in I}b_i$.
It follows that $\mathbf 1_b$ is a linear combination of the functions
$\mathbf 1_{b_I}$  ($I\subseteq [k]$) which belong to $X$ if $I\neq\emptyset$ (as $X$ is 
closed under $\wedge$ operation) and equal $\mathbf 1$, otherwise.
\qed\end{proof}

We are coming to the final transformation of our route: 
One can see that bounded additive real-value functions on a Boolean algebra $B$ naturally define
continuous linear forms on the vector space $V(B)$ hence, by density, on the Banach space $C(S(B))$ (of 
all continunous functions on $S(B)$ equipped with supremum norm).
It follows (see e.g. \cite{Rudin1973}) from  Riesz representation theorem that the topological dual of 
$C(S(B))$ is the space ${\rm rca}(S(B))$ of all regular countably additive measures on $S(B)$. 
Thus the equivalence of ${\rm ba}(B)$ and
${\rm rca}(S(B))$ follows. We summarize all of this as the following:

\begin{proposition}
\label{prop:barca}
Let $B$ be a Boolean algebra, let ${\rm ba}(B)$ be the Banach space of bounded additive real-valued functions 
equipped with the norm $\|f\|=\sup_{b\in B} f(b)-\inf_{b\in B} f(b)$, let $S(B)$ be the Stone space associated 
to $B$ by Stone representation theorem, 
and let ${\rm rca}(S(B))$ 
be the Banach space of the regular countably additive measure on $S(B)$ equipped with
the total variation norm.

Then the mapping $C_K: {\rm rca}(S(B))\rightarrow {\rm ba}(B)$ defined by $C_K(\mu)=\mu\circ K$ is an isometric isomorphism.
In other words, $C_K$ is defined by
$$C_K(\mu)(b)=\mu(\{x\in S(B):\ b\in x\})$$ (considering that the points of $S(B)$ are the ultrafilters on $B$).
\end{proposition}

Note also that, similarly, the 
restriction of $C_K$ to the space ${\rm Pr}(S(B))$ of all (regular) probability measures on $S(B)$ 
is an isometric isomorphism of ${\rm Pr}(S(B))$ and the 
subset ${\rm ba}_1(B)$ of ${\rm ba}(B)$ of all positive additive functions
$f$ on $B$ such that $f(1)=1$.

A standard notion of convergence in ${\rm rca}(S(B))$ 
(as the continuous dual of $C(S(B))$) is the weak $*$-convergence: 
a sequence $(\mu_n)_{n\in\bbbn}$ of measures
is convergent if, for every $f\in C(S(B))$ the sequence $\int f(x)\,{\rm d}\mu_n(x)$ is convergent.
Thanks to the density of $V(B)$
 this convergence translates as pointwise convergence in ${\rm ba}(B)$ as follows:
a sequence $(g_n)_{n\in\bbbn}$ of functions in ${\rm ba}(B)$ is convergent if, for every $b\in B$ the sequence $(g_n(b))_{n\in\bbbn}$ is convergent.
As ${\rm rca}(S(B))$ is complete, so is ${\rm rca}(B)$. Moreover, it is easily checked that 
${\rm ba}_1(B)$ is closed in ${\rm ba}(B)$.

In a more concise way, we can write, for a sequence $(f_n)_{n\in\bbbn}$ of functions in ${\rm ba}(B)$ and
for the corresponding sequence $(\mu_{f_n})_{n\in\bbbn}$ of regular measures on $S(B)$:
$$
\lim_{n\rightarrow\infty} f_n\text{ pointwise}\qquad\iff\qquad \mu_{f_n}\Rightarrow\mu_f.
$$
The whole situation is summarized on Fig.~\ref{fig:topdual}.
\begin{figure}[h!t]
\begin{center}
\includegraphics[height=.4\textheight]{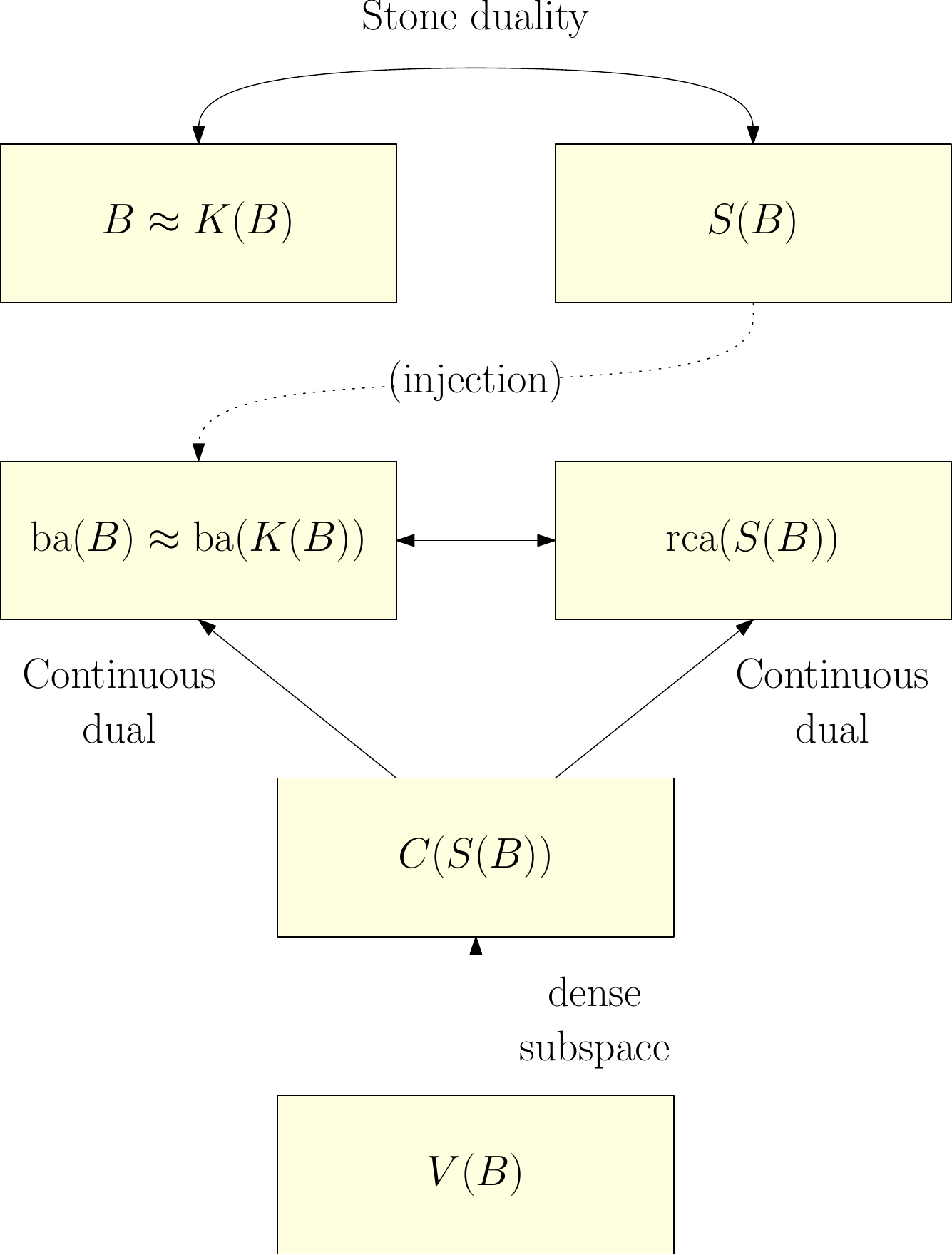}
\end{center}
\caption{Several spaces defined from a Boolean algebra, and their inter-relations.}
\label{fig:topdual}
\end{figure}

The above theory was not developed for its own sake but in order to demonstrate a natural approach to structural limits.
The next example is a continuation of our main interpretation, which we started in Logical examples~\ref{lex:1} and~\ref{lex:2}.

\begin{lexample}
Let $B={\rm FO}_0(\mathcal L)$ denote the Boolean algebra of all first-order sentences on a language $\mathcal L$ 
up to logical equivalence. We already noted that the points of $S(B)$
are complete theories of ${\rm FO}_0(\mathcal L)$, and that
each complete theory has at least one model.
Assume $\mathcal L$ is a finite language. Then for every $n\in\bbbn$
there exists a sentence $\phi_n$ such that for every complete theory $T\in{\rm FO}_0(\mathcal L)$
it holds $\phi_n\in T$ if and only if $T$ has a unique model and this model has at most $n$ elements.  
Let $U=\bigcup_{n\geq 1}K(\phi_n)$. Then $U$ is open but not closed.
The indicator function $\mathbf 1_U$ is thus measurable but not continuous.
This function has the nice property that for every complete theory $T\in S(B)$
it holds
$$
\mathbf 1_U(T)=\begin{cases}
1,&\text{if }T\text{ has a finite model;}\\
0,&\text{otherwise.}
\end{cases}
$$
\end{lexample}

\section{Limits via Fragments and Measures}
\label{sec:logic}
We provide a unifying approach based on the previous section. 
We consider the special case of Boolean algebras induced by a fragment of the class ${\rm FO}(\mathcal L)$ of
the first-order formulas over a finite relational language $\mathcal L$. In this context, the language $\mathcal L$
will be described by its {\em signature}, that is the set of non-logical symbols (constant symbols, and relation symbols,
 along with the arities of the relation symbols). An ${\rm FO}(\mathcal L)$-{\em structure}
is then a set together with an interpretation of all relational and function symbols. Thus for example
the signature of the language $\mathcal L^G$ of graphs is the symbol $\sim$ interpreted as the adjacency relation:
$x\sim y$ if $\{x,y\}$ is an edge of the graph.

We now introduce our notion of convergence.
Our approach is a combination of model theoretic and analytic approach. 

Recall that a formula is obtained from atomic formulas by the use of the negation ($\neg$), 
logical connectives ($\vee$ and $\wedge$), and quantification ($\exists$ and $\forall$).
A {\em sentence} (or closed formula) is a formula without free variables.

The {\em quantifier rank} ${\rm qrank}(\phi)$ of a formula $\phi$ is the maximum depth of
a quantifier in $\phi$. 
For instance, the quantifier rank of the formula
$$
\exists x\ ((\exists y\ (x\sim y)) \vee (\forall y\ \forall z\ \neg(x\sim y)\wedge\neg(y\sim z)))
$$ 
has quantifier rank $3$.

 The key to our approach is the following definition.

\begin{definition}
 Let $\phi(x_1,\dots,x_p)$ be a first-order formula with $p$ free variables 
 (in the language $\mathcal L$) and
 let $G$ be an  $\mathcal L$-structure. We denote
 
 \begin{equation}
 \label{eq:def}
 \langle \phi,G\rangle=\frac{|\{(v_1,\dots,v_p)\in G^p:\
 G\models\phi(v_1,\dots,v_p)\}|}{|G|^p}.
 \end{equation}
\end{definition}

 In other words, $\langle\phi,G\rangle$ is the probability that 
 $\phi$ is satisfied in $G$ when the $p$ free variables
 correspond to a random $p$-tuple of vertices of $G$. 
 The value $\langle\phi,G\rangle$ is called the {\em density} of $\phi$ in $G$.
 Note that this definition
 is consistent in the sense that although any formula $\phi$ with $p$ free
 variables can be considered as a formula with $q\geq p$ free variables with
 $q-p$ unused variables, we have
 
 $$\frac{|\{(v_1,\dots,v_q):\
 G\models\phi(v_1,\dots,v_p)\}|}{|G|^q}=\frac{|\{(v_1,\dots,v_p):\
 G\models\phi(v_1,\dots,v_p)\}|}{|G|^p}.
 $$
 
It is immediate that for every formula $\phi$ it holds
$\langle\neg\phi,G\rangle=1-\langle\phi,G\rangle$. Moreover, 
if $\phi_1,\dots,\phi_n$ are formulas, then by de Moivre's formula, it holds
 
 $$\langle\bigvee_{i=1}^n\phi_i,G\rangle=\sum_{k=1}^n
 (-1)^{k+1}\biggl(\sum_{1\leq i_1<\dots<i_k\leq n}\langle
 \bigwedge_{j=1}^k\phi_{i_j},G\rangle\biggr).$$
 
In particular, if $\phi_1,\dots,\phi_k$ are {\em mutually exclusive}
 (meaning that $\phi_i$ and $\phi_j$ cannot hold simultaneously for $i\neq j$)
 then it holds
$$\langle\bigvee_{i=1}^k\phi_i,G\rangle=\sum_{i=1}^k\langle \phi_i,G\rangle.$$
 
In particular, for every fixed graph $G$, the mapping $\phi\mapsto\langle\phi,G\rangle$ is additive
(i.e. $\langle\,\cdot\,,G\rangle\in{\rm ba}({\rm FO}(\mathcal L))$):
$$
\phi_1\wedge\phi_2=0\quad\Longrightarrow\quad \langle\phi_1\vee\phi_2,G\rangle=\langle\phi_1,G\rangle
+\langle\phi_2,G\rangle.
$$ 

Thus we may apply the above theory to additive functions $\langle\,\cdot\,,G\rangle$ and to structural limits we shall define now.

Advancing this note that  
in the case of a sentence $\phi$ (that is a formula with no free variables, i.e. $p=0$), the
definition reduces to 
$$\langle \phi,G\rangle=\begin{cases}
1,&\text{if }G\models\phi;\\
0,&\text{otherwise.}
\end{cases}.$$

Thus the definition of $\langle\phi,G\rangle$ will suit to the elementary convergence.
Elementary convergence and all above graph limits are captured by the following definition:

\begin{definition}
\label{def:converge}
Let $X$ be a fragment of ${\rm FO}(\mathcal L)$.

A sequence $(G_n)_{n\in\bbbn}$ of $\mathcal L$-structures is
{\em $X$-convergent} if for every $\phi\in X$, the sequence   
$(\langle\phi,G_n\rangle)_{n\in\bbbn}$ converges.
\end{definition}

For a Boolean sub-algebra $X$ of ${\rm FO}(\mathcal L)$, we define $\mathfrak T(X)$ has the
space of all ultrafilters  on $X$, which we call 
{\em complete $X$-theories}. The space $\mathfrak T(X)$ is endowed 
with the topology defined from its clopen sets, which are defined as the sets
$K(\phi)=\{T\in \mathfrak T(X): T\ni\phi\}$ for some $\phi\in X$. 
In the sake for simplicity, we denote by $\mathbf 1_\phi$ (for $\phi\in X$) the
indicator function of the clopen set $K(\phi)$ defined by $\phi$. Hence,
$\mathbf 1_\phi(T)=1$ if $\phi\in T$, and $\mathbf 1_\phi(T)=0$ otherwise.

It should be now clear that the above general approach yields the following:

\begin{theorem}
\label{thm:bamu}
Let $X$ be a Boolean sub-algebra of ${\rm FO}(\mathcal L)$ and let $\mathcal G$ be the class of all finite 
$\mathcal L$-structures.

There exists an injective mapping $G\mapsto\mu_G$ from $\mathcal G$ to the space of probability measures on $\mathfrak T(X)$ such that
for every $\phi\in X$ it holds
$$
\langle \phi, G\rangle=\int \mathbf 1_{\phi}(T)\,{\rm d}\mu_G(T).
$$ 
A sequence $(G_n)_{n\in\bbbn}$ of finite $\mathcal L$-structures is $X$-convergent if and only if the sequence
$(\mu_{G_n})_{n\in\bbbn}$ is weakly convergent. Moreover, if $\mu_{G_n}\Rightarrow\mu$ then for every $\phi\in X$ it holds
$$
\lim_{n\rightarrow\infty}\langle \phi, G_n\rangle=\int \mathbf 1_{\phi}(T)\,{\rm d}\mu(T).
$$   
\end{theorem}

In this paper, we shall be interested in specific fragments of ${\rm FO}(\mathcal L)$:
\begin{itemize}
  \item ${\rm FO}(\mathcal L)$ itself;
  \item ${\rm FO}_p(\mathcal L)$ (where $p\in\bbbn$), which is the fragment consisting of
  all formulas with at most $p$ free variables (in particular, ${\rm FO}_0(\mathcal L)$ is
  the fragment of all first-order sentences);
  \item ${\rm QF}(\mathcal L)$, which is the fragment of {\em quantifier-free formulas} (that is: propositional logic);
  \item ${\rm FO}^{\rm local}(\mathcal L)$, which is the fragment of {\em local formulas}, defined as follows.
\end{itemize}
Let $r\in\bbbn$. A formula $\phi(x_1,\dots,x_p)$ is {\em $r$-local} if, for
every $\mathcal L$-structure $G$ and every $v_1,\dots,v_p\in G^p$ it holds
$$
G\models\phi(v_1,\dots,v_p)\quad\iff\quad
G[N_r(v_1,\dots,v_p)]\models\phi(v_1,\dots,v_p),$$
where $N_r(v_1,\dots,v_p)$ is the closed $r$-neighborhood of $x_1,\dots,x_p$ in the
$\mathcal L$-structure $G$ (that is the set of elements at distance at most $r$ from at least one
of $x_1,\dots,x_p$ in the Gaifman graph of $G$), and where $G[A]$ denotes the sub-$\mathcal L$-structure
of $G$ induced by $A$.
A formula $\phi$ is {\em local} if it is $r$-local for some $r\in\bbbn$;
the fragment ${\rm FO}^{\rm local}(\mathcal L)$ is the set of all
local formulas (over the language $\mathcal L$).
This fragment form an important fragment, particularly because of the following structure theorem. 

\begin{theorem}[Gaifman locality theorem \cite{Gaifman1982}]
\label{thm:gaifman}
For every first-order formula $\phi(x_1,\dots,x_n)$ there exist integers
$t$ and $r$ such that $\phi$ is equivalent to a Boolean combination of $t$-local
formulas $\xi_j(x_{i_1},\dots,x_{i_s})$ and sentences of the form

\begin{equation}
\label{eq:gl}
\exists
y_1\dots\exists y_m\biggl(\bigwedge_{1\leq i<j\leq m}{\rm
dist}(y_i,y_j)>2r\wedge\bigwedge_{1\leq i\leq
m}\psi(y_i)\biggr)
\end{equation}

where $\psi$ is $r$-local. Furthermore, if
$\phi$ is a sentence, only sentences~\eqref{eq:gl} occur in the Boolean
combination.
\end{theorem}

From this theorem follows a general statement:

\begin{proposition}
\label{prop:fole}
Let $(G_n)$ be a sequence of graphs. Then $(G_n)$ is ${\rm FO}$-convergent if
and only if it is both ${\rm FO}^{\rm local}$-convergent and elementarily-convergent.
\end{proposition}
\begin{proof}
Assume $(G_n)_{n\in\bbbn}$ is both ${\rm FO}^{\rm local}$-convergent and elementarily-convergent and
let $\phi\in{\rm FO}$ be a first order formula with $n$ free variables.
According to Theorem~\ref{thm:gaifman}, there exist integers
$t$ and $r$ such that $\phi$ is equivalent to a Boolean combination of $t$-local
formula $\xi(x_{i_1},\dots,x_{i_s})$ and of sentences. It follows that
$\langle\phi,G\rangle$ can be expressed as a function of
values of the form $\langle\xi,G\rangle$ where $\xi$ is either a local formula
or a sentence. Thus $(G_n)_{n\in\bbbn}$ is ${\rm FO}$-convergent.
\qed\end{proof}

Notice that if $\phi_1$ and $\phi_2$ are local formulas, so are $\phi_1\wedge\phi_2$,
$\phi_1\vee\phi_2$ and $\neg\phi_1$. It follows that ${\rm FO}^{\rm local}$ is a Boolean sub-algebra
of ${\rm FO}$. It is also clear that all the other fragments described above correspond 
to sub-algebras of ${\rm FO}$. This means that there exist  
canonical injective Boolean-algebra homomorphisms from these
fragments $X$ to ${\rm FO}$, that will correspond to surjective
continuous functions (projections) from $S({\rm FO})$ to $S(X)$
and it is not hard to see that they also correspond to surjective maps
from ${\rm ba}({\rm FO})$ to ${\rm ba}(X)$ and to surjective 
pushforwards from ${\rm rca}(S({\rm FO})$ to ${\rm rca}(S(X))$.

Recall that a {\em theory} $T$ is a set of sentences.
(Here we shall only consider first-order theories, so a theory is a set of first-order
sentences.) The theory $T$ is {\em consistent} if one cannot deduce from $T$
both a sentence $\phi$ and its negation.
The theory $T$ is 
{\em satisfiable} if it has a model. It follows from G\"odel's completeness
theorem that, in the context of first-order logic, a theory is consistent if and only if
it is satisfiable. 
Also, according to the compactness theorem,
a theory has a model if and only if every finite subset of it has a model.
 Moreover, according to the downward L\"owenheim-Skolem theorem, there exists a countable model.
A theory $T$ is a {\em complete theory} if it is consistent
and if, for every sentence $\phi\in{\rm FO}_0(\mathcal L)$, either $\phi$ or $\neg\phi$ belongs to $T$.
Hence every complete theory has a countable model. However, a complete theory
which has an infinite model has infinitely many non-isomorphic models.

It is natural to ask whether one can consider fragments 
that are not Boolean sub-algebras of ${\rm FO}(\mathcal L)$
and still have a description of the limit 
of a converging sequence as a probability measure on a nice
measurable space.
There is obviously a case where this is possible: when the 
convergence of $\langle\phi,G_n\rangle$ for every $\phi$ in 
a fragment $X$ implies the convergence of $\langle\psi,G_n\rangle$ 
for every $\psi$ in the minimum Boolean algebra containing $X$.
We prove now that this is an instance of a more general phaenomenon:
 
\begin{proposition}
\label{prop:extBA}
Let $X$ be a fragment of ${\rm FO}(\mathcal L)$ closed under (finite) 
conjunction --- that is: a {\em meet semilattice} of ${\rm FO}(\mathcal L)$ --- and let ${\rm BA}(X)$ 
be the Boolean algebra generated by $X$ (that is the closure of $X$ by $\vee,\wedge$ and $\neg$).
Then $X$-convergence is equivalent to ${\rm BA}(X)$-convergence. 
\end{proposition}
\begin{proof}
Let $\Psi\in {\rm BA}(X)$. 
According to Lemma~\ref{lem:meetbasis}, there exist
$\phi_1,\dots,\phi_k\in X$ and $\alpha_0,\alpha_1,\dots,\alpha_k\in\bbbr$ such that

$$
\mathbf 1_\Psi=\alpha_0\mathbf 1+\sum_{i=1}^k\alpha_i\mathbf 1_{\phi_i}.
$$
Let $G$ be a graph, let $\Omega=S({\rm BA}(X))$ and let $\mu_G\in{\rm rca}(\Omega)$ 
be the associated measure. Then
$$
\langle\Psi,G\rangle=\int_{\Omega}\mathbf 1_\Psi\,{\rm d}\mu_G
=\int_{\Omega}\bigl(\alpha_0\mathbf 1+\sum_{i=1}^k\alpha_i\mathbf 1_{\phi_i}\bigr)\,{\rm d}\mu_G
=\alpha_0+\sum_{i=1}^k\alpha_i\langle\phi_i,G\rangle.
$$

Thus if $(G_n)_{n\in\bbbn}$ is an $X$-convergent sequence,
the sequence $(\langle\psi,G_n\rangle)_{n\in\bbbn}$ converges for every
$\psi\in{\rm BA}(X)$, that is $(G_n)_{n\in\bbbn}$ is ${\rm BA}(X)$-convergent.
\qed\end{proof}

Continuing to develop the general mechanism for the structural limits we consider fragments of FO quantified by the number
of free variables.

We shall allow formulas with $p$ free variables to be considered as a formula with $q>p$ variables, 
$q-p$ variables being unused. As the order of the free variables in the definition of the formula is primordial,
it will be easier for us to consider sentences with $p$ constants instead of formulas with $p$ free variables.
Formally, denote by $\mathcal L_p$ the language obtained from $\mathcal L$ by adding $p$ (ordered) symbols of constants
$c_1,\dots,c_p$. 
There is a natural isomorphism of Boolean algebras $\nu_p:{\rm FO}_p(\mathcal L)\rightarrow {\rm FO}_0(\mathcal L_p)$,
which replaces the occurrences of the $p$ free variables $x_1,\dots,x_p$  
in a formula $\phi\in{\rm FO}_p$ by the corresponding symbols of constants
$c_1,\dots,c_p$, so that it holds, for every graph $G$, for every $\phi\in{\rm FO}_p$ and
every $v_1,\dots,v_p\in G$:

$$
G\models\phi(v_1,\dots,v_p)\quad\iff\quad (G,v_1,\dots,v_p)\models\nu_p(\phi).
$$

The Stone space associated to the Boolean algebra ${\rm FO}_0(\mathcal L_p)$ is
the space $\mathfrak T(\mathcal L_p)$ of all complete theories in the language $\mathcal L_p$.
Also, we denote by $\mathfrak T_\omega$ the Stone space representing
the Boolean algebra $\mathfrak T({\rm FO}_0(\mathcal L_\omega))\approx {\rm FO}$.
One of the specific properties of the spaces $\mathfrak T(\mathcal L_p)$ is that they are endowed with  
an ultrametric derived from the quantifier-rank:
$$
{\rm dist}(T_1,T_2)=\begin{cases}
0&\text{if }T_1=T_2\\
2^{-\min\{{\rm qrank}(\theta):\ \theta\in T_1\setminus T_2\}}&\text{otherwise}.
\end{cases}
$$
This ultrametric defines the same topology as the Stone representation theorem. As a compact metric space,
$\mathfrak T(\mathcal L_p)$ is (with the Borel sets defined by the metric topology) 
a standard Borel space.

For each $p\geq 0$, there is a natural projection $\pi_p:\mathfrak
T_{p+1}\rightarrow\mathfrak T_p$, which maps a complete theory $T\in\mathfrak T_{p+1}$
to the subset of $T$ containing the sentences where only the $p$ first constant
symbols $c_1,\dots,c_p$ are used. Of course we have to check that $\pi_p(T)$ is a complete theory in 
the language $\mathcal L_p$ but this is indeed so. 

According to the ultrametrics defined above, 
the projections $\pi_p$ are
contractions (hence are continuous). Also, there is a natural isometric
embedding $\eta_p:\mathfrak T_p\rightarrow\mathfrak T_{p+1}$ defined as follows:
for $T\in\mathfrak T_p$, the theory $\eta_p(T)$ is the deductive closure of
$T\cup\{c_p=c_{p+1}\}$. Notice that $\eta_p(T)$ is indeed complete: for every
sentence $\phi\in{\rm FO}(\mathcal L_{p+1})$, let $\widetilde{\phi}$ be the
sentence obtained from $\phi$ by replacing each symbol $c_{p+1}$ by $c_p$. It is
clear that $c_p=c_{p+1}\vdash \phi\leftrightarrow\widetilde{\phi}$. As either
$\widetilde{\phi}$ or $\neg\widetilde{\phi}$ belongs to $T$, either $\phi$ or
$\neg\phi$ belongs to $\eta_p(T)$. Moreover, we deduce easily from the fact
that $\widetilde{\phi}$ and $\phi$ have the same quantifier rank that $\eta_p$ is an
isometry. Finally, let us note that $\pi_p\circ\eta_p$ is the
identity of $\mathfrak T_p$.

For these fragments we shall show a particular nice construction, well non-standard construction, of limiting measure.

\section{A Non-standard Approach}
\label{sec:nonstd}
The natural question that arises from the result of the previous
section is whether one can always find a representation of the
${\rm FO}$-limit of an ${\rm FO}$-converging sequence by a
``nice'' measurable $\mathcal L$-structure. 

It appears that a general notion of limit object for ${\rm FO}$-convergence can be
obtained by a non-standard approach. In this we follow closely Elek and Szegedy~\cite{ElekSze}. 

We first recall the ultraproduct construction.
Let $(G_n)_{n\in\bbbn}$ be a finite sequence of finite $\mathcal L$-structures and let $U$ be a non-principal 
ultrafilter.
 Let $\widetilde{G}=\prod_{i\in\bbbn}G_i$ and let $\sim$ be the equivalence
 relation on $\widetilde{V}$ defined by $(x_n)\sim(y_n)$ if $\{n: x_n=y_n\}\in
 U$. Then the {\em ultraproduct} of the $\mathcal L$-structures $G_n$ is the quotient of
 $\widetilde{G}$ by $\sim$, and it is denoted $\prod_U G_i$. 
 For each relational symbol $R$ with arity $p$, the interpretation $R^{\widetilde{G}}$ of $R$ in the ultraproduct
 is defined by 
 $$
 ([v^1],\dots,[v^p])\in R^{\widetilde{G}}\quad\iff\quad\{n: (v^1_n,\dots,v^p_n)\in R^{G_n}\}\in U. 
 $$

The fundamental theorem of ultraproducts proved by {\L}o\'s makes ultraproducts
particularly useful in model theory. We express it now in the particular case of
$\mathcal L$-structures indexed by $\bbbn$ but its general statement concerns structures indexed
by a set $I$ and the ultraproduct constructed by considering an ultrafilter $U$
over $I$.

\begin{theorem}[\cite{Los1955a}]
\label{thm:los}
For each formula $\phi(x_1,\dots,x_p)$ and each $v^1,\dots,v^p\in\prod_i G_i$ we
have
$$
\prod_U G_i\models\phi([v^1],\dots,[v^p])\quad\text{ iff }\quad
\{i:\ G_i\models\phi(v^1_i,\dots,v^p_i)\}\in U.
$$
\end{theorem}

Note that if $(G_i)$ is elementary-convergent, then $\prod_U G_i$ is an
elementary limit of the sequence: for every sentence $\phi$, according to
Theorem~\ref{thm:los}, we have
$$
\prod_U G_i\models\phi\quad\iff\quad \{i:\ G_i\models\phi\}\in U.
$$

A measure $\nu$ extending the normalised counting measures $\nu_i$ of $G_i$
is then obtained via the Loeb measure construction.  We denote by $\mathcal P(G_i)$ the
Boolean algebra of the subsets of vertices of $G_i$, with the normalized measure
$\nu_i(A)=\frac{|A|}{|G_i|}$. We define $\mathcal P=\prod_i\mathcal P(G_i)/I$,
where $I$ is the ideal of the elements $\{A_i\}_{i\in\bbbn}$ such that
$\{i:\ A_i=\emptyset\}\in U$.
We have
$$
[x]\in[A]\quad\text{iff}\quad \{i:\ x_i\in A_i\}\in U.
$$
These sets form a Boolean algebra over $\prod_U G_i$. 
Recall that the ultralimit $\lim_U a_n$ defined for every
$(a_n)_{n\in\bbbn}\in\ell^\infty(\bbbn)$ is such that for every $\epsilon>0$ we
have 

$$
\{i:\ a_i\in[\lim_U a_n-\epsilon\,;\, \lim_U a_n+\epsilon]\}\in U.
$$

Define
$$
\nu([A])=\lim_U \nu_i(A_i).
$$

Then $\nu:\mathcal P\rightarrow\bbbr$ is a finitely additive measure.
Remark that, according to Hahn-Kolmogorov theorem, proving that $\nu$ extends
to a countably additive measure amounts to prove that for every sequence $([A^n])$
of disjoint elements of $\mathcal P$ such that $\bigcup_n [A^n]\in\mathcal P$ it holds
$\nu(\bigcup_n [A^n])=\sum_n\nu([A^n])$. 

A subset $N\subseteq \prod_U G_i$ is a {\em nullset} if for every $\epsilon>0$
there exists $[A^\epsilon]\in \mathcal P$ such that $N\subseteq [A^\epsilon]$
and $\nu([A^\epsilon])<\epsilon$. The set of nullsets is denoted by $\mathcal
N$. A set $B\subseteq \prod_U G_i$ is {\em measurable} if there exists
$\widetilde{B}\in\mathcal P$ such that $B\Delta\widetilde{B}\in\mathcal N$.

The following theorem is proved in \cite{ElekSze}:
\begin{theorem}
The measurable sets form a $\sigma$-algebra $B_U$ and
$\nu(B)=\nu(\widetilde{B})$ defines a probability measure on $B_U$.
\end{theorem} 
Notice that this construction extends to the case where to each $G_i$ is
associated a probability measure $\nu_i$. Then the limit measure $\nu$ is
non-atomic if and only if the following technical condition holds:
for every $\epsilon>0$ and for every $(A_n)\in\prod G_n$, if 
for $U$-almost all $n$ it holds $\nu_n(A_n)\geq\epsilon$ then there exists
$\delta>0$ and $(B_n)\in\prod G_n$ such that for $U$-almost all $n$ it holds
$B_n\subseteq A_n$ and $\min(\nu_n(B_n),\nu_n(A_n\setminus B_n))\geq\delta$. 
This obviously holds if $\nu_n$ is a normalized counting measure and
$\lim_U|G_n|=\infty$. 

Let $f_i:G_i\rightarrow[-d;d]$ be real functions, where $d>0$. One can define
$f:\prod_U G_i\rightarrow [-d;d]$ by 
$$f([x])=\lim_U f_i(x_i).$$
We say that $f$ is the {\em ultralimit} of the functions $\{f_i\}_{i\in\bbbn}$
and that $f$ is an {\em ultralimit function}.

Let $\phi(x)$ be a first order formula with a single free variable, and let 
$f^\phi_i:G_i\rightarrow\{0,1\}$ be defined by 
$$f^\phi_i(x)=\begin{cases}
1&\text{if }G_i\models\phi(x);\\
0&\text{otherwise.}
\end{cases}$$
and let $f^\phi:\prod_U G_i\rightarrow\{0,1\}$ be defined similarly on the $\mathcal L$-structure 
$\prod_U G_i$.
Then $f^\phi$ is the ultralimit of the functions $\{f^\phi_i\}$ according to
Theorem~\ref{thm:los}.

The following lemma is proved in \cite{ElekSze}.
\begin{lemma}
\label{lem:mulim}
The ultralimit functions are measurable on $\prod_U G_i$ and
$$
\int_{\prod_U G_i} f\, {\rm d}\nu=\lim_U\frac{\sum_{x\in G_i}f_i(x)}{|G_i|}.
$$
\end{lemma}
In particular, for every formula $\phi(x)$ with a single free variable, we have:
$$
\nu\bigl(\bigl\{[x]: \prod_U
G_i\models\phi([x])\bigr\}\bigr)=\lim_U\langle\phi,G_i\rangle.
$$

Let  $\psi(x,y)$ be a formula with two free
variables.
Define $f_i:G_i\rightarrow[0;1]$ by 
$$
f_i(x)=\frac{|\{y\in G_i:\ G_i\models\psi(x,y)\}|}{|G_i|}.
$$
and let
$$f([x])=\mu\bigl(\bigl\{[y]: \prod_U G_i\models\psi([x],[y]\bigr\}\bigr).$$
Let us check that $f([x])$ is indeed the ultralimit of $f_i(x_i)$.
Fix $[x]$. Let $g_i:G_i\rightarrow\{0,1\}$ be defined by
$$
g_i(y)=\begin{cases}
1&\text{if }G_i\models \psi(x_i,y)\\
0&\text{otherwise.}
\end{cases}
$$
and let $g:\prod_U G_i\rightarrow \{0,1\}$ be defined similarly by
$$
g([y])=\begin{cases}
1&\text{if }\prod_U G_i\models \psi([x],[y])\\
0&\text{otherwise.}
\end{cases}
$$

According to Theorem~\ref{thm:los} we have
$$
\prod_U G_i\models\psi([x],[y])\quad\iff\quad\{i:\ G_i\models\psi(x_i,y_i)\}\in
U.$$
It follows that $g$ is the ultralimit of the functions $\{g_i\}_{i\in\bbbn}$.
Thus, according to Lemma~\ref{lem:mulim} we have
$$
\nu\bigl(\bigl\{[y]: \prod_U G_I\models\psi([x],[y])\bigr\}\bigr)=
\lim_U \frac{|\{y\in G_i: G_i\models\psi(x_i,y_i)\}|}{|G_i|},
$$ 
that is:
$$
f([x])=\lim_U f_i(x_i).
$$
Hence $f$ is the ultralimit of the functions $\{f_i\}_{i\in\bbbn}$ and,
according to Lemma~\ref{lem:mulim}, we have 

$$
\iint 1_\psi([x],[y])\ {\rm d}\nu([x])\ {\rm
d}\nu([y])=\lim_U\langle\psi,G_i\rangle. $$

This property extends to any number of free variables. We formulate this as a summary of the results of this section.
\begin{proposition}
\label{prop:hlim}
Let $(G_n)_{n\in\bbbn}$ be a sequence of finite $\mathcal L$-structures and let $U$ be a non-principal ultrafilter
on $\bbbn$. Then there exists a measure $\nu$ on the ultraproduct $\widetilde{G}=\prod_U G_n$ such that
for every first-order formula $\phi$ with $p$ free
variables it holds:

$$
\idotsint\limits_{\widetilde{G}^p} \mathbf 1_\phi([x_1],\dots,[x_p])\ {\rm d}\nu([x_1])\,\dots\,{\rm
d}\nu([x_p])=\lim_U\langle\psi,G_i\rangle.$$
Moreover, the above integral is invariant by any permutation on the order of the integrations:
for every permutation $\sigma$ of $[p]$ it holds
$$\lim_U\langle\psi,G_i\rangle=
\idotsint\limits_{\widetilde{G}^p} 1_\phi([x_1],\dots,[x_p])\ {\rm d}\nu([x_{\sigma(1)}])\,\dots\,{\rm
d}\nu([x_{\sigma(p)}]).$$
\end{proposition}

However, the above constructed measure algebra is non-separable (see \cite{ElekSze,Conley2012} for discussion). 

\section{A Particular Case}
\label{sec:case}
Instead of restricting convergence to a fragment of ${\rm FO}(\mathcal L)$,
it is also interesting to consider restricted classes of structures. 
For instance, the class of graphs with maximum degree at most $D$ (for some integer $D$)
received much attention. Specifically, the notion of {\em local weak convergence} of bounded degree graphs was introduced in
\cite{Benjamini2001}:

A {\em rooted graph} is a pair $(G, o)$, where $o\in V(G)$.
 An {\em isomorphism} of rooted graph $\phi: (G, o)\rightarrow (G', o')$
  is an isomorphism of the underlying
graphs which satisfies $\phi(o) = o'$. 
Let $D\in\bbbn$. 
Let ${\mathcal G}_D$ denote the collection of
all isomorphism classes of connected 
rooted graphs with maximal degree at most $D$.
For simplicity's sake, we denote elements of $\mathcal G_D$ simply as graphs.
 For $(G, o)\in{\mathcal G}_D$ and
$r\geq 0$ let $B_G(o, r)$ 
denote the subgraph of $G$ spanned by the vertices at 
distance at most $r$ from
$o$. If $(G, o), (G', o') \in{\mathcal G}_D$ and $r$
 is the largest integer such that
$(B_G(o, r), o)$ is
rooted-graph isomorphic to
$(B_{G'}(o', r), o')$, then set 
$\rho((G, o), (G', o'))= 1/r$, say. Also take 
$\rho((G, o), (G, o))=0$.
 Then $\rho$ is metric on ${\mathcal G}_D$. 
 Let $\mathfrak M_D$ denote the space of all probability measures
 on ${\mathcal G}_D$  that
are measurable with respect to the Borel $\sigma$-field of 
$\rho$. Then $\mathfrak M_D$ is endowed
with the topology of weak convergence, and is compact in this topology.

A sequence $(G_n)_{n\in\bbbn}$ of finite connected graphs with maximum degree at most $D$ is 
{\em BS-convergent} if, for every integer $r$ and every rooted connected graph 
$(F,o)$ with maximum degree at most $D$ the following limit exists:
$$
\lim_{n\rightarrow\infty}\frac{|\{v: B_{G_n}(v,r)\cong (F,o)\}|}{|G_n|}.
$$

This notion of limits leads to the definition of a limit object
as a probability measure on ${\mathcal G}_D$ \cite{Benjamini2001}.

However, as we shall
see below, a nice representation
of the limit structure can be given.
To relate BS-convergence to $X$-convergence, we shall consider the fragment 
${\rm FO}_1^{\rm local}$ of those formulas with 
at most $1$ free variable that are local. Formally, let
${\rm FO}_1^{\rm local}={\rm FO}^{\rm local}\cap{\rm
FO}_1$.

\begin{proposition}
\label{prop:BS}
Let $(G_n)$ be a sequence of finite graphs with maximum degree $d$, with
$\lim_{n\rightarrow\infty}|G_n|=\infty$. 

Then the following properties are equivalent:
\begin{enumerate}
  \item the sequence $(G_n)_{n\in\bbbn}$ is BS-convergent;
  \item the sequence $(G_n)_{n\in\bbbn}$ is ${\rm FO}_1^{\rm
local}$-convergent;
  \item the sequence $(G_n)_{n\in\bbbn}$ is ${\rm FO}^{\rm
local}$-convergent.
\end{enumerate}
\end{proposition}
\begin{proof}
If $(G_n)_{n\in\bbbn}$ is ${\rm FO}^{\rm local}$-convergent, it is ${\rm FO}^{\rm
local}_1$-convergent;

If $(G_n)_{n\in\bbbn}$ is ${\rm FO}_1^{\rm local}$-convergent then it is 
BS-convergent as for any finite rooted graph $(F,o)$, testing whether the
the ball of radius $r$ centered at a vertex $x$ is isomorphic to
$(F,o)$ can be formulated by a local first order formula.

Assume $(G_n)_{n\in\bbbn}$ is BS-convergent.
As we consider graphs with maximum degree $d$, there are only
finitely many isomorphism types for the balls of radius $r$ centered at a
vertex. It follows that any local formula $\xi(x)$ with a single variable can be
expressed as the conjunction of  a finite number of (mutually exclusive)
formulas $\xi_{(F,o)}(x)$, which in turn correspond to subgraph testing.
It follows that BS-convergence implies ${\rm FO}_1^{\rm local}$-convergence.

Assume $(G_n)_{n\in\bbbn}$ is ${\rm FO}_1^{\rm local}$-convergent and
let $\phi(x_1,\dots,x_p)$ be an $r$-local formula.
Let $\mathcal F_\phi$ be the set of all $p$-tuples 
$((F_1,f_1),\dots,(F_p,f_p))$ of rooted connected graphs
with maximum degree at most $d$ and radius (from the root) at most $r$ such that
$\bigcup_i F_i\models \phi(f_1,\dots,f_p)$.

Then, for every graph $G$ the sets

$$\{(v_1,\dots,v_p):\ G\models\phi(v_1,\dots,v_p)\}$$
and
$$\biguplus_{((F_1,f_1),\dots,(F_p,f_p))\in\mathcal F_\phi}\prod_{i=1}^p
\{v:\ G\models\theta_{(F_i,f_i)}(v)\}$$
differ by at most $O(|G|^{p-1})$ elements.
Indeed, according to the definition of an $r$-local formula, the $p$-tuples
$(x_1,\dots,x_p)$ belonging to exactly one of these sets are such that there
exists $1\leq i<j\leq p$ such that ${\rm dist}(x_i,x_j)\leq 2r$.

It follows that
$$
\langle\phi,G\rangle=\bigl(\sum_{((F_i,f_i))_{1\leq i\leq p}\in\mathcal
F_\phi}\, \prod_{i=1}^p\,\langle\theta_{(F_i,f_i)},G\rangle\bigr)+O(|G|^{-1}).
$$

It follows that ${\rm FO}^{\rm local}_1$-convergence (hence BS-convergence)
implies full ${\rm FO}^{\rm local}$-convergence.
\qed\end{proof}

According to this proposition, the BS-limit of a sequence of graphs
with maximum degree at most $D$ corresponds to a probability measure
on $S({\rm FO}_1^{\rm local}(\mathcal L))$ 
(where $\mathcal L$ is the language of graphs) whose support is included in
the clopen set $K(\zeta_D)$,  where $\zeta_D$ is the sentence
expressing that the maximum degree is at most $D$.
As above, the Boolean algebra ${\rm FO}_1^{\rm local}(\mathcal L)$ is isomorphic
to the Boolean algebra defined by the fragment 
$X\subset {\rm FO}_0(\mathcal L_1)$ 
of sentences in the language of rooted graphs
that are local with respect to the root.  
According to this locality, any two countable rooted graphs $(G_1,r_1)$ 
and $(G_2,r_2)$,
the trace of the complete theories of $(G_1,r_1)$ and $(G_2,r_2)$ on $X$
 are the same if and only if the (rooted) connected component $(G_1',r_1)$ of $(G_1,r_1)$ 
 containing the root $r_1$ is elementary equivalent to the (rooted) connected component $(G_2',r_2)$ of $(G_2,r_2)$ 
 containing the root $r_2$. As isomorphism and elementary equivalence
are equivalent for  countable connected graphs with bounded degrees 
it is easily checked that $K_X(\zeta_D)$ is homeomorphic to $\mathcal G_D$.
Hence our setting leads essentially to the same limit object as \cite{Benjamini2001}
for BS-convergent sequences. 

We now consider how full ${\rm FO}$-convergence differs to BS-convergence
for sequence of graphs with maximum degree at most $D$. This shows a remarkable 
stability of BS-convergence.
\begin{corollary}
\label{cor:BSE}
A sequence $(G_n)$ of finite graphs with maximum degree at most $d$ such that
$\lim_{n\rightarrow\infty}|G_n|=\infty$ is
${\rm FO}$-convergent if and only if it is both 
BS-convergent and elementarily convergent.
\end{corollary}
\begin{proof}
This is a direct consequence of Propositions~\ref{prop:fole} and~\ref{prop:BS}.
\qed\end{proof}

Explicit limit objects are known for sequence of bounded degree graphs, both for
BS-convergence (graphing) and for elementary convergence (countable graphs).
It is natural to ask whether a nice limit object could exist for full ${\rm
FO}$-convergence. We shall now answer this question by the positive.

Let $V$ be a standard Borel space with a measure $\mu$.
Suppose that $T_1, T_2,\dots, T_k$ are
measure preserving Borel involutions of $X$.
Then the system $$\mathbf G = (V, T_1, T_2, \dots, T_k, \mu)$$ is
called a {\em measurable graphing} \cite{Adams1990}.
Here $x$ is adjacent to $y$, if $x \neq y$ and $T_j(x) = y$
 for some $1\leq j \leq k$. Now if
If $V$ is a compact metric space with a Borel measure $\mu$ and
$T_1, T_2,\dots, T_k$ are continuous
measure preserving involutions of $V$, then
$\mathbf G = (V, T_1, T_2, \dots, T_k, \mu)$ is
a {\em topological graphing}.
It is a consequence of \cite{Benjamini2001} and \cite{Gaboriau2005} that
every local weak limit of finite connected graphs with maximum
degree at most $D$ can be represented
as a measurable graphing. Elek \cite{Elek2007b} further proved the representation
can be required to be a topological graphing.

For an integer $r$, a graphing ${\mathbf G}=(V,T_1,\dots,T_k,\mu)$ and
 a finite rooted graph $(F,o)$ we define the set 

$$D_r(\mathbf G,(F,o))=\{x\in\mathbf G, B_r(\mathbf G,x)\simeq (F,o)\}.$$

We shall make use of the following lemma which reduces a graphing to its essential support.

\begin{lemma}[Cleaning Lemma]
Let $\mathbf G=(V, T_1,\dots,T_d, \mu)$ be a graphing.

Then there exists a subset $X\subset V$ with $0$ measure such that
$X$ is globally invariant by each of the $T_i$ and
$\mathbf G'=(V-X,T_1,\dots,T_d, \mu)$ is a graphing
such that  for every finite rooted graph $(F,o)$ and integer $r$ it holds
$$
\mu(D_r(\mathbf G',(F,o)))=\mu(D_r(\mathbf G,(F,o)))
$$
(which means that $\mathbf G'$ is equivalent to $\mathbf G$) and
$$D_r(\mathbf G',(F,o))\neq\emptyset\quad\iff\quad\mu(D_r(\mathbf G',(F,o)))>0.$$
\end{lemma}
\begin{proof}
For a fixed $r$, define $\mathcal F_r$ has the set of all (isomorphism types of)
finite rooted graphs $(F,o)$ with radius at most $r$ such that 
$\mu(D_r(\mathbf G,(F,o)))=0$.
Define
$$
X=\bigcup_{r\in\bbbn}\bigcup_{(F,o)\in \mathcal F_r}D_r(\mathbf G,(F,o)).
$$
Then $\mu(X)=0$, as it is a countable union of $0$-measure sets. 

We shall now prove that $X$ is a union of connected components of $\mathbf G$, that is that
$X$ is globally invariant by each of the $T_i$.
Namely, if $x\in X$ and $y$ is adjacent to $x$, then $y\in X$.
Indeed: if $x\in X$ then there exists an integer $r$ such that
$\mu(D(\mathbf G,B_r(\mathbf G,x)))=0$. But it is easily checked that

$$
\mu(D(\mathbf G,B_{r+1}(\mathbf G,y)))\leq d \cdot\mu(D(\mathbf G,B_r(\mathbf
G,x))).
$$
Hence $y\in X$.
It follows that for every $1\leq i\leq d$ we have $T_i(X)=X$. 
So we can define the graphing $\mathbf G'=(V-X,T_1,\dots,T_d,\mu)$.

Let $(F,o)$ be a rooted finite graph. Assume there exists
$x\in\mathbf G'$ such that $B_r(\mathbf G',r)\simeq (F,o)$. As $X$ is a union
of connected components, we also have $B_r(\mathbf G,r)\simeq (F,o)$ and
$x\notin X$. 

It follows that $\mu(D(\mathbf G,(F,o)))>0$  hence
$\mu(D_r(\mathbf G',(F,o)))>0$.
\qed\end{proof}

The cleaning lemma allows us a clean description of ${\rm FO}$-limits in the bounded
degree case:

\begin{theorem}
\label{thm:Bounded}
Let $(G_n)_{n\in\bbbn}$ be a FO-convergent sequence of finite graphs with
maximum degree $d$, with $\lim_{n\rightarrow\infty}|G_n|=\infty$. Then
there exists a graphing $\mathbf G$ and a countable graph $\hat G$ such that
\begin{itemize}
  \item $\mathbf G$ is a BS-limit of the sequence,
  \item $\hat G$ is an elementary limit of the sequence,
  \item $\mathbf G\cup\hat G$ is an ${\rm FO}$-limit of the sequence.
\end{itemize}
\end{theorem}
\begin{proof}
Let $\mathbf G$ be a BS-limit, which has been ``cleaned'' using the previous
lemma, and let $\hat G$ be an elementary limit of $G$. It is clear that 
$\mathbf G\cup\hat G$ is also a BS-limit of the sequence, so the lemma amounts
in proving that $\mathbf G\cup\hat G$ is elementarily equivalent to $\hat G$.

According to Hanf's theorem \cite{Hanf1965}, it  is sufficient to prove that for
every integers $r,t$ and for every rooted finite graph $(F,o)$ (with
maximum degree $d$) the following equality holds:

$$\min(t,|D_r(\mathbf G\cup\hat G,(F,o))|)=\min(t,|D_r(\hat G,(F,o))|).
$$
Assume for contradiction that this is not the case. Then 
$|D_r(\hat G,(F,o))|<t$ and $D_r(\mathbf G,(F,o))$ is not empty. However, as
$\mathbf G$ is clean, this implies $\mu(D_r(\mathbf G,(F,o)))=\alpha>0$.
It follows that for every sufficiently large $n$ it holds
$|D_r(G_n,(F,o))|>\alpha/2\, |G_n|>t$. Hence $|D_r(\hat G,(F,o))|>t$, contradicting
our hypothesis.

Note that the reduction of the satisfaction problem of a general first-order formula
$\phi$ with $p$ free variables to a case analysis based on the isomorphism type 
of a bounded neighborhood of the free variables shows that every first-order
definable subset of $(\mathbf G\cup\hat G)^p$ is indeed measurable (we extend $\mu$ to
$\mathbf G\cup\hat G$ in the obvious way, considering $\hat G$ as zero measure).
\qed\end{proof}

The cleaning lemma sometimes applies in a non-trivial way:
\begin{example}
Consider the graph $G_n$ obtained from a De Bruijn sequence (see e.g. \cite{Invitation1998}) of
length $2^n$ as shown Fig~\ref{fig:debruijn}.

\begin{figure}[h!t]
\begin{center}
\includegraphics[width=.5\textwidth]{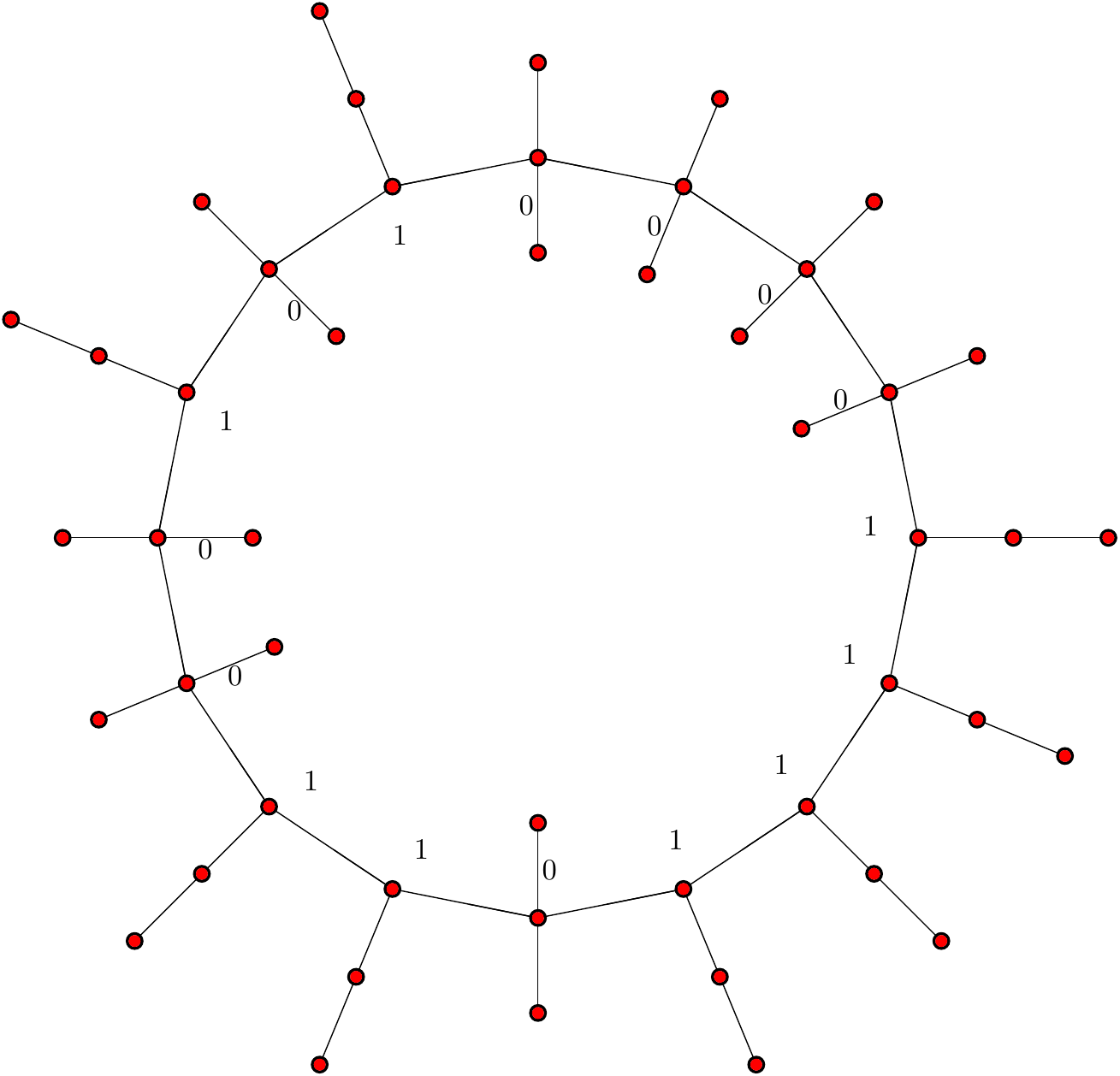}
\end{center}
\caption{The graph $G_n$ is constructed from a De Bruijn sequence of length
$2^n$.}
\label{fig:debruijn}
\end{figure}

It is easy to define a graphing $\mathbf G$, which is the limit of the sequence
$(G_n)_{n\in\bbbn}$: as vertex set, we consider the rectangle
$[0;1)\times[0;3)$. We define a measure preserving function $f$ and two
measure preserving involutions $T_1,T_2$ as follows:
\begin{align*}
f(x,y)&=\begin{cases}
(2x,y/2)&\text{if }x<1/2\text{ and }y<1\\
(2x-1,(y+1)/2)&\text{if }1/2\leq x\text{ and }y<1\\
(x,y)&\text{otherwise}
\end{cases}\\
T_1(x,y)&=\begin{cases}
(x,y+1)&\text{if }y<1\\
(x,y-1)&\text{if }1\leq y<2\\
(x,y)&\text{otherwise}
\end{cases}\\
T_2(x,y)&=\begin{cases}
(x,y+1)&\text{if }x<1/2\text{ and }1\leq y<2\\
(x,y+2)&\text{if }1/2\leq x\text{ and }y<1\\
(x,y-1)&\text{if }x<1/2\text{ and }2\leq y\\
(x,y-2)&\text{if }1/2\leq x\text{ and }2\leq y\\
(x,y)&\text{otherwise}
\end{cases}\\
\end{align*}

Then the edges of $\mathbf G$ are the pairs $\{(x,y),(x',y')\}$ such that
$(x,y)\neq(x',y')$ and either $(x',y')=f(x,y)$, or $(x,y)=f(x',y')$, or
$(x',y')=T_1(x,y)$, or $(x',y')=T_2(x,y)$.

If one considers a random root $(x,y)$ in $\mathbf G$, then the connected
component of $(x,y)$ will almost surely be a rooted line with some decoration,
as expected from what is seen from a random root in a sufficiently large $G_n$.
However, special behaviour may happen when $x$ and $y$ are rational. Namely, it
is possible that the connected component of $(x,y)$ becomes finite. For
instance, if $x=1/(2^n-1)$ and $y=2^{n-1}x$ then the orbit of $(x,y)$ under the
action of $f$  has length $n$ thus the connected component of $(x,y)$ in
$\mathbf G$ has order $3n$. Of course, such finite connected components do
not appear in $G_n$.
Hence, in order to clean $\mathbf G$, infinitely many components have to be
removed.
\end{example}

\section{Conclusion and Further Works}
\label{sec:conc}
In a forthcoming paper \cite{NPOM1}, we apply the theory developed here to
the context of classes of graphs with bounded diameter connected
components, and in particular to classes with bounded tree-depth~\cite{Taxi_tdepth}.
Specifically, we prove that if a uniform bound is fixed on the diameter of the connected
components, ${\rm FO}$-convergence may be considered component-wise (up to some 
residue for which ${\rm FO}_1$-convergence is sufficient).

The prototype of convenient limit objects for sequences of finite graphs 
is a quadruple $\mathbb G=(V,E,\Sigma,\mu)$, 
where $(V,E)$ is a graph, $(V,\Sigma,\mu)$ is a standard probability space, and
$E$ is a measurable subset of $V^2$. In such a context, modulo the axiom of projective determinacy
(which would follow from the existence of infinitely many Woodin cardinals \cite{Martin1989}),
every first-order definable subset of $V^p$ is measurable in $(V^p,\Sigma^{\otimes p})$ \cite{Mycielski1964}.
Then, for every first-order formula $\phi$ with $p$ free variables, it is natural to define
$$\langle\psi,\mathbb G\rangle=\int_{V^p} 1_\phi\ {\rm d}\mu^{\otimes p}.$$
In this setting, $\mathbb G=(V,E,\Sigma,\mu)$ is a limit --- we don't pretend to have uniqueness --- of an
${\rm FO}$-convergent sequence $(G_n)_{n\in\bbbn}$ of finite graphs if
for every first-order formula $\psi$ it holds
$$
\langle\psi,\mathbb G\rangle=\lim_{n\rightarrow\infty}\langle\psi, G_n\rangle.
$$
We obtain in \cite{NPOM1} an explicit construction of such limits 
for ${\rm FO}$-convergent sequences of finite graphs bound to a class of graphs with bounded tree-depth.
It is also there where we develop in a greater detail the general theory 
explained in the Sections~\ref{sec:bool} and~\ref{sec:logic}.
Notice that in some special cases, one does not need a standard probability space and
a Borel measurable space is sufficient. This is for instance the case when we consider limits
of finite connected graphs with bounded degrees (as we can use a quantifier elimination scheme to prove
that definable sets are measurable) or quantifier-free convergence of graphs (definable sets form indeed
a sub-algebra of the $\sigma$-algebra).

\section*{Acknowledgments}

The authors would like to thanks the referee for his most valuable comments.
\providecommand{\noopsort}[1]{}\providecommand{\noopsort}[1]{}

\end{document}